\documentclass[11pt,english]{article}

\usepackage[margin = 2.5cm]{geometry}

\usepackage{amsmath}
\usepackage{amsthm}
\usepackage{amssymb}
\usepackage{mathtools}
\usepackage{graphicx}
\usepackage{xcolor}
\usepackage[pdftex,colorlinks,backref=page,citecolor=blue,bookmarks=false]{hyperref}
\usepackage{enumitem}
\usepackage{bbm}

\usepackage{caption}% http://ctan.org/pkg/caption
\usepackage[square,sort,comma,numbers]{natbib} 
\usepackage{setspace}
\usepackage{cleveref}
\theoremstyle{plain}

\newtheorem*{theorem*}{Theorem}
\newtheorem{theorem}{Theorem}[section]
\crefname{theorem}{Theorem}{Theorems}
\Crefname{theorem}{Theorem}{Theorems}

\newtheorem*{lemma*}{Lemma}
\newtheorem{lemma}[theorem]{Lemma}
\crefname{lemma}{Lemma}{Lemmas}
\Crefname{lemma}{Lemma}{Lemmas}

\newtheorem*{claim*}{Claim}
\newtheorem{claim}[theorem]{Claim}
\crefname{claim}{Claim}{Claims}
\Crefname{claim}{Claim}{Claims}

\crefname{proposition}{Proposition}{Propositions}
\Crefname{proposition}{Proposition}{Propositions}

\newtheorem{corollary}[theorem]{Corollary}
\crefname{corollary}{Corollary}{Corollaries}
\Crefname{corollary}{Corollary}{Corollaries}

\crefname{conjecture}{Conjecture}{Conjectures}
\Crefname{conjecture}{Conjecture}{Conjectures}

\newtheorem{question}[theorem]{Question}
\crefname{question}{Question}{Questions}
\Crefname{question}{Question}{Questions}

\newtheorem{observation}[theorem]{Observation}
\crefname{observation}{Observation}{Observations}
\Crefname{observation}{Observation}{Observations}

\crefname{example}{Example}{Examples}
\Crefname{example}{Example}{Examples}

\theoremstyle{definition}

\crefname{problem}{Problem}{Problems}
\Crefname{problem}{Problem}{Problems}

\newtheorem{definition}[theorem]{Definition}
\crefname{definition}{Definition}{Definitions}
\Crefname{definition}{Definition}{Definitions}

\theoremstyle{remark}

\crefname{remark}{Remark}{Remarks}
\Crefname{remark}{Remark}{Remarks}

\usepackage{xpatch}
\xpatchcmd{\proof}{\itshape}{\normalfont\proofnamefont}{}{}
\newcommand{\proofnamefont}{}
\renewcommand{\proofnamefont}{\bfseries}

\usepackage{framed}

\newcommand{\remove}[1]{}

\newcommand{\R}{\mathbb{R}}

\def\abhi#1{}
\def\tao#1{}
\def\liana#1{}
\def\shoham#1{}

\let\abhi=\abhiOpt % COMMENT OUT for clean output
\let\tao=\taoOpt % COMMENT OUT for clean output
\let\liana=\lianaOpt % COMMENT OUT for clean 
\let\shoham=\shohamOpt % COMMENT OUT for clean

\newcommand{\eps}{\varepsilon}
\newcommand{\lam}{\lambda}

\DeclareMathOperator{\ex}{ex}
\DeclareMathOperator{\bad}{bad}

\newcommand{\homstar}{\hom^*}
\DeclareMathOperator{\Hom}{Hom}
\newcommand{\Homstar}{\Hom^*}

\newcommand{\homwstar}{\homw^{*}}

\newcommand{\exstar}{\ex^*}
\newcommand{\C}{\mathcal{C}}
\newcommand{\rhomin}{\rho_{\min}}
\newcommand{\rhomax}{\rho_{\max}}

\newcommand{\one}{\mathbbm{1}}
\newcommand{\HH}{\mathcal{H}}

		\def \homw {\hom_{\omega}}
		
		\def \homws {\hom_{\omega}^*}
		\def \Deltaw {\Delta_{\omega}}
		
		\def \maxw {\omega_{\max}}
		\def \minw {\omega_{\min}}

\renewcommand{\Pr}{\mathbb{P}}

\title{Rainbow subdivisions of cliques}

\author{
        Tao Jiang 
        \thanks{Department of Mathematics, Miami University, Oxford, OH 45056, USA. Email: \texttt{jiangt}@\texttt{miamioh.edu}.
        Research supported by National Science Foundation grant DMS-1855542.}
    \and
	    Shoham Letzter\thanks{
		Department of Mathematics, 
		University College London, 
		Gower Street, London WC1E~6BT, UK. 
		Email: \texttt{s.letzter}@\texttt{ucl.ac.uk}. 
		Research supported by the Royal Society.
    }
    \and
        Abhishek Methuku%
        \thanks{School of Mathematics,
        University of Birmingham,
        Birmingham, B15 2TT, UK. Email: \texttt{abhishekmethuku}@\texttt{gmail.com}.  Research supported by the EPSRC, grant no. EP/S00100X/1 (A. Methuku).}
    \and
     	Liana Yepremyan\thanks{ 
    	Deparment of Mathematics, Emory University, Atlanta, US.
    	Email: \texttt{lyeprem}@\texttt{emory.edu}. Research supported by Marie Sklodowska Curie Global Fellowship, H2020-MSCA-IF-2018:846304.
	}
}

\begin{document}

\date{}
\maketitle

\begin{abstract}

	\setlength{\parskip}{\medskipamount}
    \setlength{\parindent}{0pt}
    \noindent
    We show that for every integer $m \ge 2$ and large $n$, every properly edge-coloured graph on $n$ vertices with at least $n (\log n)^{53}$ edges contains a rainbow subdivision of $K_m$. This is sharp up to a polylogarithmic factor. Our proof method exploits the connection between the mixing time of random walks and expansion in graphs.
\end{abstract}

\section{Introduction} \label{sec:intro}

    The \emph{Tur\'an number} of a graph $H$, denoted by $\ex(n, H)$, is the maximum possible number of edges in an $n$-vertex graph that does not contain a copy of $H$. In this paper we study a rainbow variant of Tur\'an numbers, introduced by Keevash, Mubayi, Sudakov and Verstra\"ete \cite{keevash2007rainbow}. A \emph{proper edge-colouring} of a graph is an assignment of colours to its edges so that edges that share a vertex have distinct colours. A \emph{rainbow} subgraph of an edge-coloured graph is a subgraph whose edges have distinct colours.
    The \emph{rainbow Tur\'an number} of a graph $H$, denoted by $\exstar(n, H)$, is the maximum possible number of edges in a properly edge-coloured graph on $n$ vertices with no rainbow copy of $H$. One can define $\ex(n, \HH)$ and $\exstar(n, \HH)$ analogously for a family of graphs $\HH$.
    
    It was shown in \cite{keevash2007rainbow} that
    $\exstar(n, H) = (1 + o(1))\ex(n, H)$ for non-bipartite $H$. Perhaps unsurprisingly, little is known about rainbow Tur\'an numbers of bipartite graphs. The authors of \cite{keevash2007rainbow} raised two problems concerning rainbow Tur\'an numbers of even cycles, one concerning an even cycle of fixed length $2k$ and the other concerning the family $\C$ of all cycles. 
    For all $k\geq 2$,  they showed that $\exstar(n, C_{2k}) = \Omega(n^{1 + 1/k})$ and conjectured that
    $\exstar(n, C_{2k}) = \Theta(n^{1 + 1/k})$. The authors of \cite{keevash2007rainbow} verified the conjecture
    for $k \in \{2, 3\}$. Following further progress on the conjecture by Das, Lee and Sudakov \cite{das2013rainbow},  Janzer \cite{janzer2020rainbow} recently resolved the conjecture. 
    
    Regarding the rainbow Tur\'an number of the family $\C$ of all cycles,
    Keevash, Mubayi, Sudakov and Verstra\"ete \cite{keevash2007rainbow} showed that $\exstar(n, \C) = \Omega(n \log n)$, by considering a naturally defined proper edge-colouring of the hypercube $Q_k$, where $k=\lfloor \log n \rfloor$ (colour the edge $uv$ by colour $i$ if the $u$ and $v$ differ in the coordinate $i$; it is easy to see that no cycle can be rainbow in this colouring). They also showed that $\exstar(n, \C) = O(n^{4/3})$ and
    asked if $\exstar(n, \C) = O(n^{1+o(1)})$ and furthermore, if $\exstar(n, \C) = O(n \log n)$.
    Das, Lee and Sudakov \cite{das2013rainbow} answered the first question affirmatively, by showing that $\exstar(n, \C)  \le ne^{(\log n)^{\frac{1}{2}+o(1)}}$. In 2022, Janzer \cite{janzer2020rainbow} improved this bound by establishing that $\exstar(n, \C) = O\big(n (\log n)^4\big)$, which is tight up to a polylogarithmic factor. Recently, Jiang, Methuku and Yepremyan \cite{jiang2021rainbow} proved the following generalisation of Das, Lee and Sudakov~\cite{das2013rainbow} on $\exstar(n,\C)$.
    %Unlike  \cite{das2013rainbow}, 

    \begin{theorem}[Jiang, Methuku, Yepremyan~\cite{jiang2021rainbow}]
    \label{thm:JMY}
        For every integer $m \geq 2$ there exists a constant $c>0$ such that for every integer $n\geq m$ the following holds. If $G$ is a properly edge-coloured graph  on $n$ vertices with at least $ne^{c\sqrt{\log n}}$ edges, then $G$ contains a rainbow subdivision of $K_m$, where each edge is subdivided at most $1300 \log^2 n$ times.
    \end{theorem}
    The method used in \cite{jiang2021rainbow} utilises robust expanders in the coloured setting together with a density increment argument, inspired in part by the method introduced by Sudakov and Tomon~\cite{sudakov2020extremal}.
    
    In this paper, we lower the $e^{O(\sqrt{\log n})}$ error term in Theorem~\ref{thm:JMY} to a polylogarithmic term, which in conjunction with the above-mentioned $\Omega(n\log n)$ lower bound on $\exstar(n,\C)$ determines the rainbow Tur\'an number of the family of $K_m$-subdivisions up to a polylogarithmic factor.

    \begin{theorem}\label{thm:main}
        Fix an integer $m \geq 2$ and let $n$ be sufficiently large.
        Suppose that $G$ is a properly edge-coloured graph on $n$ vertices with at least $n (\log n)^{53}$ edges. Then $G$ contains a rainbow subdivision of $K_m$, where each edge is subdivided at most $(\log n)^{6}$ times.
    \end{theorem}
    
     Theorem~\ref{thm:main} provides the rainbow analogue of a fundamental (and highly influential) result of Mader~\cite{mader1967homomorphieeigenschaften} stating that for every integer $m \ge 2$, there exists $d = d(m)$ such that every graph with average degree at least $d$ contains a subdivision of $K_m$. Research on this problem has a long history, see e.g., Mader~\cite{mader1972hinreichende}, Koml\'os and Szemer\'edi~\cite{komlos1994topological,komlos1996topological}, and Bollob\'as and Thomason~\cite{bollobas1998proof}.

    Our proof of Theorem~\ref{thm:main} exploits the connection between mixing time of random walks and edge expansion.
    This connection is used in conjunction with counting lemmas developed by Janzer in \cite{janzer2020rainbow} regarding homomorphisms of cycles in graphs.
    We also prove a strengthening of \Cref{thm:main}, regarding `rooted' rainbow subdivisions of $K_m$ in expanders (see \Cref{thm:main-rooted}). For this stronger version, in addition to the ingredients used for proving Theorem~\ref{thm:main}, we use the framework of \cite{jiang2021rainbow} and an additional idea used by Letzter in  \cite{letzter2021tight} 
    (see Lemma~\ref{lem:reachable-robust-new}).

    The rest of the paper is organised as follows. In  \Cref{sec:overview},
    we give a short overview of our proofs, namely the proof of \Cref{thm:main} and a strengthening of it (\Cref{thm:main-rooted}).
    In \Cref{sec:prelims}, we mention various preliminary results, regarding the existence of expanders which are close to being regular and properties of expanders. In \Cref{sec:non-rainbow-walks}, we state three lemmas due to Janzer \cite{janzer2020rainbow} and some consequences of these lemmas. \Cref{sec:random-walks} contains the main new ideas of the paper, exploiting a connection between the mixing time of a random walk and expansion properties in a graph. In \Cref{sec:main-proof},
    we prove \Cref{thm:main} and a strengthening of it regarding rooted subdivisions in almost regular expanders. We complete the paper with concluding remarks in \Cref{sec:conc}.
    
    Throughout the paper, for convenience, we drop floor and ceiling signs for large numbers, and logarithms are in base $2$.

\section{Overview of the proofs} \label{sec:overview}
    Our main idea is to use the connection between the mixing time of random walks, the notion of `conductance' (see \Cref{def:conductance}) and our notion of expansion.     
    It is a well-known and very useful fact that `large' conductance implies `small' mixing time (see, e.g., Lov\'asz \cite{lovasz1993random}). Moreover, our notion of expansion implies that our expanders have large conductance. Using these facts we show that if additionally, such expanders are almost regular, then long enough walks are close to being uniformly distributed.     
    We also use two counting lemmas of Janzer from~\cite{janzer2020rainbow}. Below we describe these lemmas and the main ideas in more detail.
    
    In a properly edge-coloured graph, say that a closed walk is \emph{degenerate} if it is either not rainbow or visits a vertex more than once. 
    The first lemma from \cite{janzer2020rainbow} implies that in a properly edge-coloured graph which is close to being regular, the number of degenerate closed $2k$-walks is significantly smaller than the number of closed $2k$-walks, provided that $k$ is sufficiently large.
    
    Given two vertices $x$ and $y$, a closed $2k$-walk $W$ is said to \emph{be hosted} by $x$ and $y$ if it starts at $x$ and reaches $y$ after $k$ steps. We call a pair of vertices $(x, y)$ \emph{good} if the number of degenerate closed $2k$-walks hosted by $x$ and $y$ is significantly smaller than the number of closed $2k$-walks hosted by $x$ and $y$. The second lemma from \cite{janzer2020rainbow} that we use shows that if a pair $(x, y)$ is good then there are many short pairwise colour-disjoint and internally vertex-disjoint $k$-paths from $x$ to $y$.

	In fact, we use versions of these lemmas which are applicable to edge-weighted graphs. These weighted versions can be easily deduced for the original unweighted versions.\footnote{To deduce the weighted version of the first lemma, we actually need a multigraph version of Janzer's original one, whose proof is identical to the original version for simple graphs.} We will later apply these weighted lemmas with a specific edge-weighting, namely where $w(xy) = \frac{1}{\sqrt{d(x)d(y)}}$. This weight was chosen so that the weight of a walk $W = x_0 \ldots x_k$ is the probability that a random walk of length $k$ starting at $x_0$ produces $W$, times $\sqrt{\frac{d(x_k)}{d(x_0)}}$.

    Using results about random walks on graphs, which relate mixing time to expansion,  we show that in an expander $G$ on $n$ vertices which is close to being regular, for $k$ suitably large (at least polylogarithmic in $n$), the numbers of closed $2k$-walks hosted by any two pairs of vertices are within a suitable polylogarithmic factor (in $n$) of each other. This, combined with the fact that the number of degenerate closed $2k$-walks is small compared to the total number of closed $2k$-walks (due to the first lemma above), implies that almost all pairs of vertices are good. Thus, using Tur\'an's theorem, we find a copy of $K_m$ in the graph formed by good pairs.
    This, together with the fact that there are many short colour-disjoint and internally vertex-disjoint rainbow paths between any good pair of vertices (due to the second lemma above) allows us to greedily build the desired rainbow-subdivision of $K_m$.

    We also prove a stronger version of Theorem~\ref{thm:main},  \Cref{thm:main-rooted}, asserting that in an expander $G$ which is close to being regular and whose average degree is large enough, for any set $S$ of $m$ vertices, there exists a rainbow $K_m$-subdivision with the vertices of $S$ being the branching vertices. The main step in this proof shows that for any two vertices $x$ and $y$ in $G$ there is a short rainbow $x,y$-path avoiding a prescribed small set $C$ of vertices and colours. By iterating this over all pairs of vertices in $S$, we can build the desired rainbow $K_m$-subdivision.
    
    To show that there is a short rainbow $x,y$-path in $G$, we first apply tools due to Jiang, Methuku and Yepremyan \cite{jiang2021rainbow} and Letzter \cite{letzter2021tight} to show that there is a set of vertices $U$ of size $\Omega(n)$ such that for each $v\in U$ there is such a short rainbow $x,v$-path $P(v)$ and a short rainbow $y,v$-path $Q(v)$, both of which avoid $C$, such that no colour is used on too many of these paths $P(v)$ and $Q(v)$. It easily follows that for almost all pairs $(u, v)$ with $u, v\in U$, the paths $P(u)$ and $Q(v)$ are colour-disjoint.
    This, combined with the fact that most pairs in $U$ are good (in the sense mentioned earlier), implies that there exists at least one good pair $(u,v)$ for which $P(u)$ and $Q(v)$ are colour-disjoint. This allows us to find a suitable short rainbow $u,v$-path $L$ such that $P(u)\cup L\cup Q(v)$ is a rainbow $x,y$-walk which contains the desired rainbow $x,y$-path.
    
    \section{Preliminaries} \label{sec:prelims}

    Let $G$ be a graph. We denote by $d(G)$ the average degree of $G$. For a subset $S \subseteq V(G)$, let $e(S) = e(G[S])$, and for subsets $S, T \subseteq V(G)$, let $e(S, T) = e(G[S, T])$.  We will use the notions of $d$-minimality and expanders, defined below, following \cite{jiang2021rainbow}.
     
    \begin{definition}
        A graph $G$ is said to be $d$-\emph{minimal} if $d(G) \geq d$ but $d(H) < d$ for every proper subgraph $H\subseteq G$.
    \end{definition}

    It is easy to see that every graph $G$ contains a $d(G)$-minimal subgraph.
    The following observation was used in \cite{jiang2021rainbow}. For completeness, we include its short proof.
        
    \begin{observation} \label{obs:d-expand}
        If $G$ is $d$-minimal, then every subset $S\subseteq V(G)$ satisfies $e(S)+e(S,S^c)\geq \frac{d |S|}{2}$.
        In particular, $\delta(G) \ge \frac{d}{2}$.
    \end{observation}
    
    \begin{proof} 
        Suppose otherwise. Then $e(S^c) \geq \frac{d|V|}{2} - \frac{d|S|}{2} \geq \frac{d(|V| - |S|)}{2}$, contradicting $d$-minimality.
    \end{proof}
    
    \begin{definition} \label{def:expander-new}
        Given $d \ge 1$, $\eta \in (0, 1)$ and $\varepsilon\in (0,\frac{1}{2}]$, an $n$-vertex graph $G$ is called a \emph{$(d, \eta, \eps)$-expander} if $G$ is $d$-minimal, and for every subset $S \subseteq V(G)$ of size at most $(1 - \eps)n$, we have $d(S) \le (1 - \eta)d$. 
    \end{definition}

    Note that, by definition, for $0<\varepsilon'\leq\varepsilon\leq\frac{1}{2}$ and $0 < \eta \le \eta' < 1$, every $(d,\eta',\varepsilon')$-expander is  also a $(d,\eta,\varepsilon)$-expander. Also, if $G$ is a $(d, \eta, \eps)$-expander then it is a $(d(G), \eta, \eps)$-expander.
    It will be useful to note the following `edge-expansion' property of $(d, \eta, \eps)$-expanders.
    
    \begin{observation} \label{obs:edge-expansion-new}
        Let $n, d \ge 1$, let $\eta\in (0, 1)$ and let $\varepsilon\in (0,\frac{1}{2}]$.
        Suppose that $G$ is a $(d, \eta, \eps)$-expander on $n$ vertices. Then every $S \subseteq V(G)$ with $|S| \le (1 - \eps)n$ satisfies $e(S, S^c) \ge \frac{\eta d}{2}|S|$.
    \end{observation}
    
    \begin{proof}
        Let $S \subseteq V(G)$ satisfy $|S| \le (1 - \eps)n$.
        Since $G$ is $d$-minimal and by \Cref{obs:d-expand}, we have $e(S) + e(S, S^c) \ge \frac{d|S|}{2}$. Since $G$ is a $(d, \eta, \eps)$-expander, by definition, we also have $e(S) = \frac{d(S)}{2}|S| \le \frac{(1 - \eta)d}{2}|S|$. It follows that $e(S, S^c) \ge \frac{\eta d}{2}|S|$, as claimed.
    \end{proof}
    
    \begin{lemma}[Lemma 2.5 from \cite{jiang2021rainbow}]
    \label{lem:existence-expanders-new}
        Let $n, d \ge 1$, let $\eps \in (0,\frac{1}{2}]$ and let $\eta = \frac{\eps}{2\log n}$.
        Suppose that $G$ is a graph on $n$ vertices with average degree $d$. Then $G$ contains a $(d', \eta, \eps)$-expander, with $d' \ge \frac{d}{2}$. 
    \end{lemma}
    
    The following lemma from \cite{jiang2021rainbow} asserts that in a properly edge-coloured expander, one can reach almost every vertex by a short rainbow path starting at a specified vertex.
    
    \begin{lemma}[Lemma 2.7 from \cite{jiang2021rainbow}] \label{lem:reachable-new}
        Let $n, \ell, d, M \ge 1$, let $\eta \in (0,1)$, and let $\eps \in (0, \frac{1}{2}]$. Suppose that $\ell = \frac{4 \log n}{\eta}$ and $d \ge \frac{4\ell + 8M}{\eta}$.
        Let $G$ be a properly edge-coloured $(d, \eta, \eps)$-expander on $n$ vertices, let $x \in V(G)$ and let $F$ be a set of vertices and colours of size at most $M$. Then at least $(1 - \eps)n$ vertices can be reached from $x$ by a rainbow path of length at most $\ell+1$ that avoids the vertices and colours in $F$.
    \end{lemma}
    
    We will need a stronger version of the previous lemma, where we require that no colour is used too many times in the short rainbow paths. A similar idea was used in \cite{letzter2021tight} (see Lemma 5) in the context of tight paths.
    
    \begin{lemma} \label{lem:reachable-robust-new}
        Let $n, \ell, d, q, M \ge 1$, let $\eta \in (0,1)$ and let $\eps \in (0, \frac{1}{2}]$.
        Suppose that $\ell = \frac{4\log n}{\eta}$ and $d \ge \frac{20q\ell + 8M}{\eta}$.
        Let $G$ be a properly edge-coloured $(d, \eta, \eps)$-expander on $n$ vertices, let $x \in V(G)$, and let $F$ be a set of colours and vertices of size at most $M$. Then there is a set $U \subseteq V(G)$ of size at least $(1 - \eps)n$, and a collection  $\mathcal P = \{ P(u) : u \in U\}$ where for each $u \in U$, $P(u)$ is a rainbow path from $x$ to $u$ of length at most $\ell + 1$ that avoids the vertices and colours in $F$ and no colour appears in more than $\frac{n}{q}$ of the paths in $\mathcal P$.
    \end{lemma}
    
    \begin{proof}
        Let $U$ be a largest set satisfying that for every $u \in U$ there is a rainbow path $P(u)$ from $x$ to $u$ of length at most $\ell + 1$ that avoids $F$, such that no colour appears in more than $\frac{n}{q}$ of the paths $P(u)$. Say that a colour is \emph{bad} if it appears on exactly $\frac{n}{q}$ of the paths $P(u)$ with $u \in U$, and let $C_{\bad}$ be the set of bad colours. Since each path $P(u)$ has length at most $\ell + 1$, we have 
        \begin{equation*}
            |C_{\bad}| \le \frac{n(\ell + 1)}{n/q} \le 2q\ell.
        \end{equation*}
        
       Since $d\ge \frac{20q\ell + 8M}{\eta}\geq \frac{4\ell+8(M+|C_{\bad})|}{\eta}$, by \Cref{lem:reachable-new} with $F\cup C_{\bad}$ playing the role of $F$, there is a set $U'$ with $|U'| \ge (1 - \eps)n$ such that for every $v \in U'$, there is a rainbow path $P'(v)$ from $x$ to $v$ of length at most $\ell + 1$ that avoids the colours and vertices in $F \cup C_{\bad}$. If $|U| < (1 - \eps)n$, then there is a vertex $v \in U' \setminus U$. The set $U \cup \{v\}$ (along with the paths $P(u)$ for $u \in U$ and $P'(v)$) contradicts the maximality of $U$. It follows that $|U| \ge (1 - \eps)n$, as required. 
    \end{proof}
    
    We would like to work with expanders that are close to regular. For this, we use the following lemma, which is a slight adaptation of Lemma 3.2 of \cite{montgomery2021C4}.
    \begin{lemma}\label{lem:regularization}
        Let $n \ge 2$ and $d \ge 36 \log n$.
        Let $G$ be a bipartite graph on $n$ vertices with minimum degree at least $d$. Then there exists a subgraph $H$ of $G$ with average degree at least $\frac{d}{12 \log{n}}$ and maximum degree at most $d$.
    \end{lemma}
        
    \begin{proof}
        Let $\{A,B\}$ denote a bipartition of $G$ with $|A|\geq |B|$. Let $G'$ be obtained from $G$ by keeping exactly $d$ edges incident to each vertex in $A$. Then for each $v\in A$ we have $d_{G'}(v)=d$, and hence $e(G')=d|A|$.
        
        Let $m=\lceil \log n\rceil$. For each $i\in [m]$, let $B_i=\{v\in B: 2^{i-1}\leq d_{G'}(v)<2^i\}$. Denoting the set of isolated vertices in $B$ by $B_0$, we have $B \setminus B_0 = \cup_{i\in [m]} B_i$.
        By the pigeonhole principle, there exists an $i\in [m]$ for which $e(G'[A,B_i])\geq \frac{e(G')}{m} \geq \frac{d|A|}{2\log n}$. 
        Fix such $i$ and let $t = 2^{i-1}$.
        Then, by definition, each $v\in B_i$ has degree between $t$ and $2t$ in $G'[A,B_i]$.
        If $2t\leq d$, then $G'[A,B_i]$ has maximum degree at most $d$ and
        average degree at least $\frac{2 e(G'[A,B_i])}{|A|+|B_i|}\geq \frac{d}{2\log n}$.
        So the lemma holds. Hence, we may assume that $2t > d$.
        
        Set $p = \frac{d}{4t}$; then $0 < p < \frac{1}{2}$. Now let $A'\subseteq A$ be chosen by including each vertex in $A$ independently with probability $p$. 
        For convenience, write $G_i=G'[A,B_i]$ and $G_i'=G'[A',B_i]$. Then
        \begin{equation} \label{A'B'}
            \mathbb{E}[e(G_i')] =p \cdot e(G_i)\geq \frac{pd|A|}{2\log n}.
        \end{equation}
        Now, let $B'=\{v\in B_i: d_{G'_i}(v)\leq d\}$. 

        For each $v\in B_i$, the degree $d_{G'_i}(v)$ is binomially distributed with expectation $p \cdot d_{G_i}(v)$. Since $t\leq d_{G_i}(v)\leq 2t$, we have
        $\frac{d}{4}=pt\leq \mathbb{E}[d_{G'_i}(v)]\leq 2pt=\frac{d}{2}$. Therefore, using Chernoff's bound (see, e.g., Appendix A of \cite{alonspencer2016}), we have
        \begin{align*}
           \mathbb{E}[|B_i\setminus B'|]&=\sum_{v\in B_i} \mathbb{P}(v\in B_i\setminus B')= \sum_{v\in B_i} \mathbb{P}\!\left[d_{G'_i}(v)\geq d\right]
            \,\leq\, \sum_{v\in B_i} \mathbb{P}\!\left[d_{G'_i}(v)\geq 2 \mathbb{E}[d_{G'_i}(v)]\right] \\
            &\,\leq\, \sum_{v\in B_i} 2\cdot \exp\!\left(-\frac{\mathbb{E}[d_{G'_i}(v)]}{3}\right) 
            \leq\ \sum_{v\in B_i}2\cdot \exp\!\left(-\frac{d}{12}\right)
            \leq\ n \cdot 2e^{-3\log n}<\frac{1}{n}.
        \end{align*}
        This together with the fact that for any $A'$ and the corresponding $B'$, $e(G'[A',B_i\setminus B']) \leq n|B_i\setminus B'|$, implies
        \begin{equation} \label{A'notB'}
        \mathbb{E}[e(G'[A',B_i\setminus B'])]\leq n \cdot \mathbb{E}[|B_i\setminus B'|]  \le 1.
        \end{equation}
        
        By \eqref{A'B'} and \eqref{A'notB'},
        \[\mathbb{E}[e(G'[A',B'])\geq \frac{pd|A|}{2\log n}-1.\]
        Note that  $\mathbb{E}[|A'|]=p|A|$. Also, as $t|B_i|\leq e(G'[A, B_i]) \leq d|A|$, we have
        $|B'|\leq |B_i|\leq \frac{d}{t}|A|=4p|A|$. Hence $|A'|+|B'|\leq 5p|A|$.
        Let $d_0=\frac{d}{12 \log n}$.
        We have
        \begin{equation*}
            \mathbb{E}\!\left[e(G'[A', B'])-d_0(|A'|+|B'|)\right]
            \geq \frac{pd|A|}{2\log n} - 1  - \frac{5dp|A|}{12 \log n} 
            \geq \frac{pd|A|}{12 \log n} - 1
            \ge 0,
        \end{equation*}
        where the last inequality used $pd|A| = \frac{d^2 |A|}{4t} \ge 18^2(\log n)^2 \ge 12 \log n$ which holds since $t \le |A|$ and $d \ge 36 \log n$.
        Thus, there is a choice of $A'$ for which $e(G'[A', B'])-d_0(|A'|+|B'|)\geq 0$. Taking $H = G'[A', B']$, we have $d(H)\geq d_0= \frac{d}{12 \log n}$ and $\Delta(H)\leq d$, as desired.
    \end{proof}
    
    Our final preliminary result combines \Cref{lem:existence-expanders-new,lem:regularization} to show that every relatively dense graph contains an expander which is close to regular.

    \begin{lemma} \label{lem:bounded-max-deg-expander}
        Let $n, d \ge 1$ and let $\eps \in (0,\frac{1}{2}]$, and suppose that $d \ge 10^7 (\log n)^3$.
        Suppose that $G$ is a bipartite graph on $n$ vertices with average degree at least $d$. Then $G$ has a subgraph $H$ with the following properties.
        \begin{enumerate} [ref = \rm\arabic*]
            \item \label{itm:property-1}
                $H$ is a $(d', \eta, \eps)$-expander on $n'$ vertices, where $d' \ge \frac{d}{2500(\log n)^2}$ and $\eta \ge \frac{\eps}{100(\log n')^2}$,
            \item \label{itm:property-2}
                $H$ has maximum degree at most $2500 (\log n')^2 d'$.
        \end{enumerate}
    \end{lemma}
    
    \begin{proof}
        Let $G_0 = G$. We run the following process, generating graphs $G_i$ for $i \ge 0$. For each graph $G_i$, we write $d_i = d(G_i)$ and $n_i = |V(G_i)|$.
        \begin{enumerate} [label = \rm(\alph*)]
            \item \label{itm:step-a}
                Let $H_i$ be a subgraph of $G_i$ with average degree at least $\frac{d_i}{24 \log n_i}$ and maximum degree at most $d_i$. Such a subgraph $H_i$ exists by \Cref{lem:regularization}, using the fact that every graph with average degree $d$ contains a subgraph with minimum degree at least $\frac{d}{2}$. To apply the lemma we need to verify that $d_i \ge 72 \log n_i$, which we shall do below.
            \item \label{itm:step-b}
                Write $n_i' = |V(H_i)|$.
                Let $G_{i+1}$ be a subgraph of $H_i$ which is a $(d_{i+1}, \eta_{i+1}, \eps)$-expander, where $d(G_{i+1}) = d_{i+1} \ge \frac{d(H_i)}{2} \ge \frac{d_i}{48 \log n_i}$ and $\eta_{i+1} = \frac{\eps}{2 \log n_i'}$. Such a subgraph exists by \Cref{lem:existence-expanders-new}, using the observation that if $G$ is a $(d, \eta, \eps)$-expander then it is a $(d(G), \eta, \eps)$-expander.
        \end{enumerate}
        
        \begin{claim}
            If $48 \log n_i < \sqrt{48 \log n_{i-1}}$ for $i \in [t]$ then the graphs $G_1, \ldots, G_t$ can be defined as above and are non-empty.
        \end{claim}
        
        \begin{proof}
            Notice that to prove the statement, it suffices to show that for $t$ as described, $d_i \ge 72 \log n_i$ for $i \in [t-1]$.
            
            We prove this by induction.
            It is easy to check that the statement is true for $t = 0$ (note that the condition on $t$ holds vacuously here). Indeed, we just need to check that $d_0 \ge 72 \log n_0$, which is the case as $d_0 = d$, $n_0 = n$ and $d \ge 10^7(\log n)^3$.
            
            Now suppose that $48 \log n_i < \sqrt{48 \log n_{i-1}}$ for $i \in [t]$ and that the inductive statement holds for $t' \le t$. This means that the process runs as described for all $i \in [t]$, and it remains to show that the $(t+1)$-st step can be performed, namely that $d_{t} \ge 72 \log n_{t}$.
            By the assumption on $t$, the following holds for every $i \in [t]$.
            \begin{equation*}
                48 \log n_i 
                \le (48 \log n_{i-1})^{1/2}
                \le \ldots \le (48 \log n_0)^{2^{-i}}
                = (48 \log n)^{2^{-i}}.
            \end{equation*}
            Since $d_i \ge \frac{d_{i-1}}{48 \log n_{i-1}}$ for $i \in [t]$ (see \ref{itm:step-b}), it follows that
            \begin{equation} \label{eqn:d}
                d_t \ge
                \frac{d_{t-1}}{48 \log n_{t-1}} 
                \ge \ldots 
                \ge \frac{d_0}{48 \log n_{t-1} \cdot \ldots \cdot 48 \log n_{0}} 
                \ge \frac{d}{(48 \log n)^{2^0 + \ldots + 2^{-(t-1)}}}
                \ge \frac{d}{(48 \log n)^2}.
            \end{equation}
            Using $d \ge 10^7 \log n$ and $n \ge n_t$, we find that $d_t \ge 10^3 \log n > 72 \log n_t$, as required.
        \end{proof}
        
        Let $\ell$ be minimum such that $48 \log n_{\ell} \ge \sqrt{48 \log n_{\ell - 1}}$. We claim that such $\ell$ exists. If not, then by the previous claim the process can be run forever and $n_i \ge 1$ for all $i \ge 0$, implying that $(48 \log n_i)_{i \ge 0}$ is decreasing, and thus $(n_i)_{i \ge 0}$ is an infinite decreasing sequence of positive integers, a contradiction. 
        
        We will show that $G_{\ell}$ satisfies the requirements of \Cref{lem:bounded-max-deg-expander}.
        Indeed, $G_{\ell}$ is a $(d_{\ell}, \eta_{\ell}, \eps)$-expander. 
        Using the proof of the above claim, inequality \eqref{eqn:d} holds for $t = \ell$, showing $d_{\ell} \ge \frac{d}{2500(\log n)^2}$.
        By choice of $G_{\ell}$, we have $\eta_{\ell} = \frac{\eps}{2\log n_{\ell-1}} \ge \frac{\eps}{96(\log n_{\ell})^2} \ge \frac{\eps}{100(\log n_{\ell})^2}$ (using $48 \log n_{\ell} \ge \sqrt{48 \log n_{\ell-1}}$). It follows that property \ref{itm:property-1} of the lemma holds.
        To see property \ref{itm:property-2}, note that $G_{\ell}$ has maximum degree at most $d_{\ell - 1}$ and $d_{\ell} \ge \frac{d_{\ell-1}}{48 \log n_{\ell-1}} \ge \frac{d_{\ell-1}}{(48 \log n_{\ell})^2} \ge \frac{d_{\ell-1}}{2500 (\log n_{\ell})^2}$ (again using $48 \log n_{\ell} \ge \sqrt{48 \log n_{\ell-1}}$).
    \end{proof}

\section{Counting rainbow cycles in weighted graphs} \label{sec:non-rainbow-walks}

	In this section we state several lemmas regarding (weighted) counts of homomorphic copies of paths and cycles in weighted graphs $G$. Three of these are weighted versions of lemmas from Janzer \cite{janzer2020rainbow}, and we will show how to deduce these from the original, unweighted versions.

	First, let us introduce some notation.
	Let $G$ be a graph with a weighting $\omega : E(G) \to \R^{\ge 0}$. We denote by $\maxw, \minw$ the maximum and minimum edge weights; that is, $\maxw=\max_{e\in E(G)}{\omega(e)}$ and $\minw=\min_{e\in E(G)}{\omega(e)}$. The \emph{weighted maximum degree of $G$}, denoted $\Deltaw(G)$, is 
	\begin{equation*}
		\Deltaw(G) := \max_{u \in V(G)} \sum_{v : \, uv \in E(G)} \omega(uv).
	\end{equation*}
	%The \emph{weight} of a subgraph $G'$ of $G$, denoted $\omega(G')$, is $\prod_{e \in E(G')} \omega(e)$.
   For any $t \ge 1$, the \emph{weight} of a walk $P\coloneqq u_0 u_1 \dots u_t $ is defined as $\omega(P) \coloneqq \prod_{i = 0}^{t-1} \omega(u_i u_{i+1})$, and the \emph{weight} of a closed walk $C \coloneqq u_0 u_1 \dots u_t u_0$ is defined as $\prod_{i = 0}^{t} \omega(u_i u_{i+1})$, where $u_{t+1} = u_0$.

    In a graph $G$, denote by $\Hom(P_k^{xy})$ the family of walks of length $k$ from $x$ to $y$. Similarly, let $\Hom(C_{2k}^{xy})$ be the family of closed walks of length $2k$ that start at $x$ and reach $y$ after $k$ steps. 
    Write $\hom(P_k^{xy})=|\Hom(P_k^{xy})|$ and $\hom(C_{2k}^{xy})=|\Hom(C_{2k}^{xy})|$. Given a weighting $\omega : E(G) \to \R^{\ge 0}$, we define the weighted homomorphism counts as follows: 
    		\begin{align*}
			\homw(P_k^{xy}) &= \sum_{P\in \Hom (P_k^{xy})}{\omega(P)}
   \\ \homw(C_{2k}^{xy}) &= \sum_{C\in \Hom(P_k^{xy})}{\omega(C)}
		\end{align*}

    The following relation between $\homw(P_k^{xy})$ and $\homw(C_{2k}^{xy})$ is very useful.
    \begin{equation} \label{eqn:hom-cycle-paths}
        \homw(C_{2k}^{xy}) = (\hom_{w}(P_k^{xy}))^2. 
    \end{equation}
    We also define $\homw(C_{2k})$ to be the total weight of the homomorphic copies of $C_{2k}$, namely 
    \begin{equation*}
        \homw(C_{2k}) = \sum_{\substack{x, y \in V(G)}} %\\ x \ne y} }
        \homw(C_{2k}^{xy}).
    \end{equation*}
    Similarly, we define $\homw(P_k) = \sum_{x, y \in V(G)} \homw(P_{k}^{xy})$.
        In a properly edge-coloured graph $G$, let $\Homstar(C_{2k}^{xy})$ be the family of all the closed walks in $\Hom(C_{2k}^{xy})$ that do not form a rainbow cycle of length $2k$.  Define $\homstar(C_{2k}^{xy})=|\Homstar(C_{2k}^{xy})|$. If the graph is weighted with a weighting $\omega : E(G) \to \R^{\ge 0}$, let  $\homwstar(C_{2k}^{xy})= \sum_{C\in \Homstar(C_{2k}^{xy})}{\omega(C)}$. Let
    \begin{align*}
        \Homstar(C_{2k}) 
        = \bigcup_{x,y\in V(G)} \Homstar(C_{2k}^{xy}),
    \end{align*} 
    and write $\homstar(C_{2k})=|\Homstar(C_{2k})|$ and $\homwstar(C_{2k})= \sum_{x,y\in V(G)}{\homwstar(C_{2k}^{xy})}$. 

    We will make use of the following two lemmas from a recent paper of Janzer \cite{janzer2020rainbow}. 

    \begin{lemma}[multigraph version of Lemma 2.2 from \cite{janzer2020rainbow}] \label{lem:janzer-vertices}
        Let $k \ge 2$ be an integer and let $G = (V, E)$ be a multigraph on $n$ vertices. Let $\sim$ be a symmetric binary relation on $V$ such that for every $u, v \in V$, there are at most $t$ edges $vw$ (counted with multiplicity) for which $u \sim w$. Then the number of homomorphic $2k$-cycles $(x_1, \ldots, x_{2k})$ in $G$ such that $x_i \sim x_j$ for some $i \neq j$ is at most
        \begin{equation*}
            32k^{3/2}t^{1/2}\Delta(G)^{1/2}n^{\frac{1}{2k}}\hom(C_{2k})^{1 - \frac{1}{2k}}.
        \end{equation*}
    \end{lemma} 
	Note that Lemma 2.2 in \cite{janzer2020rainbow} is phrased for simple graphs, where the condition was that for every $u,v\in V$ there are at most $t$ neighbours $w$ for which $u\sim w$. However, the same proof works for multigraphs with a modified condition as stated above. 
	Similarly, the next lemma is a multigraph version of Lemma 2.1 from \cite{janzer2020rainbow}, where again the same proof works.

	\begin{lemma}[multigraph version of Lemma 2.1 from \cite{janzer2020rainbow}] \label{lem:janzer-multigraph}
		Let $k \ge 2$ be an integer and let $G = (V, E)$ be a multigraph on $n$ vertices. Suppose that $\sim$ is a symmetric binary relation on $E$ such that for every $uv \in E$ and $w \in V$, there are at most $t$ edges $zw$ (counted with multiplicity) for which $uv \sim zw$. Then the number of homomorphic $2k$-cycles $(x_1, \ldots, x_{2k})$ in $G$ such that $x_i x_{i+1} \sim x_j x_{j+1}$ for some $i \neq j$ is at most
		\begin{equation*}
			32k^{3/2}t^{1/2}\Delta(G)^{1/2}n^{\frac{1}{2k}}\hom(C_{2k})^{1 - \frac{1}{2k}}.
		\end{equation*}
	\end{lemma}

	The following lemma is a variant of \Cref{lem:janzer-multigraph} which is applicable for edge-weighted graphs.
   	\begin{lemma} \label{lem:janzer-weighted}
		Let $k \ge 2$ be an integer and let $G = (V, E)$ be a graph on $n$ vertices with a given weighting $\omega : E(G) \to \R^{> 0}$. Suppose that $\sim$ is a symmetric binary relation on $E$ such that for every $uv \in E$ and $w \in V$, there are at most $t$ edges $wz$ for which $uv \sim wz$. Let $\homws(C_{2k})$ be the sum of weights of homomorphic $2k$-cycles $(x_1, \ldots, x_{2k})$ in $G$ such that $x_i x_{i+1} \sim x_j x_{j+1}$ for some $i \neq j$. Then
		\begin{equation*}
			\homws(C_{2k}) 
			\le 32k^{3/2}t^{1/2} \cdot \Deltaw(G)^{1/2} \cdot \maxw^{1/2} \cdot n^{\frac{1}{2k}} \cdot \hom_{\omega}(C_{2k})^{1 - \frac{1}{2k}}.
		\end{equation*}
	\end{lemma}

	We prove this lemma using Lemma~\ref{lem:janzer-multigraph}.\footnote{Alternatively, one can prove this by following the proof of Lemma 2.1 from \cite{janzer2020rainbow} and adapting it straightforwardly to the weighted version.}

	\begin{proof}[Proof of \Cref{lem:janzer-weighted}]
		We start by proving the lemma in the case where all the weights given by $\omega$ are integers. Let $G'$ be the multigraph on $V(G)$ where, for each edge $e$ in $G$, we add $\omega(e)$ copies of $e$ to $G'$.
		Notice that $\Delta(G') = \Deltaw(G)$.
		Let $t' := t \cdot \maxw$ and notice that for every $uv \in E(G')$ and $w \in V(G')$, there are at most $t'$ edges $zw$ such that $uv \sim zw$. 
		Let $h'=\hom(C_{2k})$ in $G'$. Observe that $\homw(C_{2k})$ in $G$ equals $h'$, and that $\homws(C_{2k})$ in $G$ is the number of homomorphic $2k$-cycles $(x_1, \ldots, x_{2k})$ in $G'$ such that $x_i x_{i+1} \sim x_j x_{j+1}$ for some $i \neq j$. Thus, by \Cref{lem:janzer-multigraph}, 
		\begin{align*}
			\homws(C_{2k}) 
			& \le 32k^{3/2}(t')^{1/2}\Delta(G')^{1/2} \cdot n^{\frac{1}{2k}} \cdot (h')^{1 - \frac{1}{2k}} \\
			& \le 32 k^{3/2} t^{1/2} \Deltaw(G)^{1/2} \cdot \maxw^{1/2} \cdot n^{\frac{1}{2k}} \cdot \homw(C_{2k})^{1 - \frac{1}{2k}}. 
		\end{align*}
		Now suppose that all weights given by $w$ are rational. Then there exists a positive integer $L$ such that the weighting $L \omega : E(G) \to \R^{\ge 0}$ (defined as $L \omega(e) = L \cdot \omega(e)$ for every $e \in E(G)$) is integer-valued. By the previous paragraph,
		\begin{align*}
			\hom_{L\omega}^*(C_{2k}) 
			& \le 32 k^{3/2} t^{1/2} \Delta_{L\omega}(G)^{1/2} \cdot ((L\omega)_{\max})^{1/2} \cdot n^{\frac{1}{2k}} \cdot \hom_{L\omega}(C_{2k})^{1 - \frac{1}{2k}}.
		\end{align*}
		Noting that $\hom_{L\omega}^*(C_{2k}) = L^{2k} \cdot \homws(C_{2k})$, $\hom_{L\omega}(C_{2k}) = L^{2k} \cdot \homw(C_{2k})$, $\Delta_{L\omega}(G) = L \cdot \Deltaw(G)$ and $(L\omega)_{\max} = L \cdot \maxw$, we get
		\begin{align*}
			\homws(C_{2k}) 
			& \le 32 k^{3/2} t^{1/2} \Deltaw(G)^{1/2} \cdot \maxw^{1/2} \cdot n^{\frac{1}{2k}} \cdot \homw(C_{2k})^{1 - \frac{1}{2k}}.
		\end{align*}
		Finally, for a general positive weighting $\omega$, let $(\omega_m)$ be a sequence of positive and rational weightings such that $\lim_{m \to \infty} \omega_m(e) = \omega(e)$ for every edge $e \in E$. Then,
		\begin{align*}
			\hom_{\omega_m}^*(C_{2k}) 
			& \le 32 k^{3/2} t^{1/2} \Delta_{\omega_m}(G)^{1/2} \cdot (\omega_m)_{\max}^{1/2} \cdot n^{\frac{1}{2k}} \cdot \hom_{\omega_m}(C_{2k})^{1 - \frac{1}{2k}},
		\end{align*}
		for every $m$. Taking the limits as $m$ tends to infinity, we get
		\begin{equation*}
			\homws(C_{2k}) 
			\le 32 k^{3/2} t^{1/2} \Deltaw(G)^{1/2} \cdot \maxw^{1/2} \cdot n^{\frac{1}{2k}} \cdot \homw(C_{2k})^{1 - \frac{1}{2k}}.
			\qedhere
		\end{equation*}
	\end{proof}
	Similarly, using Lemma~\ref{lem:janzer-vertices} instead of Lemma~\ref{lem:janzer-multigraph}, one can prove the following
	\begin{lemma} \label{lem:janzer-weighted-vertex}
		Let $k \ge 2$ be an integer and let $G = (V, E)$ be a graph on $n$ vertices with a given weighting $\omega : E(G) \to \R^{> 0}$. Suppose that $\sim$ is a symmetric binary relation on $V$ such that for every $u,v \in V$ there are at most $t$ edges $vw$ for which $u\sim w$. Let $\homws(C_{2k})$ be the sum of weights of homomorphic $2k$-cycles $(x_1, \ldots, x_{2k})$ in $G$ such that $x_i \sim x_j$ for some $i \neq j$. Then
		\begin{equation*}
			\homws(C_{2k}) 
			\le 32k^{3/2}t^{1/2} \cdot \Deltaw(G)^{1/2} \cdot \maxw^{1/2} \cdot n^{\frac{1}{2k}} \cdot \hom_{\omega}(C_{2k})^{1 - \frac{1}{2k}}.
		\end{equation*}
	\end{lemma}
  
	Next we show that in an appropriately weighted almost-regular properly edge-coloured graphs the number of degenerate homomorphic copies of $2k$-cycles is a `small' proportion of all possible copies. Here by degenerate we mean the copies which are either non-rainbow or are not isomorphic copies of $C_{2k}$.
    \begin{lemma} \label{lem:non-rainbow-hom}
        	%Let $G$ be a graph with a weighting $\omega : E(G) \to \R^{\ge 0}$. 
         Let $n, d, \mu, k, S \ge 1$ and let $\eta \in (0, 1)$. Suppose that $d \ge 2^{16} \mu^3 k^3 S^2  n^{1/k}$.
        Let $G$ be a properly edge-coloured graph with minimum degree at least $\frac{d}{2}$ and maximum degree at most $\mu d$, and define $\omega : E(G) \to \R^{\ge 0}$ by setting $\omega(xy)=1/\sqrt{d(x)d(y)}$ for every edge $xy$.
        Then $\homwstar(C_{2k}) \le \frac{1}{S} \homw(C_{2k})$.
    \end{lemma}
    
    \begin{proof}
        We first give a lower bound on $\homw(C_{2k})$. For this, note that 
		\begin{equation*}
			\homw(P_k) =\sum_{W: \,W=x_0\dots x_k}{\omega(W)} \ge n \cdot \left(\frac{d}{2}\right)^k(\minw)^k \geq  n \cdot \left(\frac{1}{2\mu}\right)^k,
		\end{equation*}
        where in the last inequality we used that $\minw \geq \frac{1}{\Delta(G)}\geq \frac{1}{\mu d}$. Hence,
        \begin{align}
			\begin{split}
				\label{eq:lowerboundhomC2k}
				\homw(C_{2k}) 
				& = \sum_{x, y \in V(G)} \big(\homw(P_k^{xy})\big)^2 \\
				& \ge \frac{1}{n^2} \left(\sum_{x, y \in V(G)}\homw(P_k^{xy}) \right)^2
			%\ge \frac{1}{n^2} \Big(\sum_{\substack{x, y \in V(G) \\ x \ne y} } \homw(P_k^{xy}) \Big)^2
				= \bigg(\frac{\homw(P_k)}{n}\bigg)^2
				\ge \left(\frac{1}{2\mu}\right)^{2k},
			\end{split}
        \end{align}
        where the first inequality follows by convexity.
        
        Let $\sim_e$ be the binary relation on $E(G)$ where $e \sim_e f$ if and only if $e$ and $f$ have the same colour.  Because $G$ is properly edge-coloured, for every edge $uv$ and vertex $w$, there is at most one edge $wz$ for which $uv \sim_e wz$.
        Let $\sim_v$ be the binary relation defined on $V(G)$ where $u \sim_v w$ if and only if $u = w$. Then, trivially, for every $u, v \in V$ there is at most one edge $vw$ (namely $uv$) for which $u \sim_v w$.
        Apply \Cref{lem:janzer-weighted,lem:janzer-weighted-vertex} with $\sim_e$, $\sim_v$ in place of $\sim$, respectively (so $t$ is taken to be $1$ in both lemmas), to obtain the desired upper bound on $\homwstar(C_{2k})$, as follows. 
        \begin{align*}
            \homwstar(C_{2k}) 
            & \le 64 k^{3/2} \cdot (\Deltaw(G))^{1/2} \cdot \maxw^{1/2} \cdot n^{\frac{1}{2k}} \cdot \homw(C_{2k})^{1 - \frac{1}{2k}} \\
            &\leq 64k^{3/2} \cdot (2\mu)^{1/2}\cdot \left(\frac{2}{d}\right)^{1/2} \cdot n^{\frac{1}{2k}} \cdot (2\mu) \cdot \homw(C_{2k})\\
            &=\frac{2^8k^{3/2}\mu^{3/2}n^{\frac{1}{2k}}}{d^{\frac{1}{2}}}\cdot \homw(C_{2k})\\
            &\le \frac{1}{S} \cdot \homw(C_{2k}). 
        \end{align*}
        In the inequalities above we used that $\maxw \leq  \frac{1}{\delta(G)} \leq \frac{2}{d}$ and $\Deltaw(G)\leq \Delta(G) \cdot \maxw \leq 2\mu$ and finally, the inequality $\homw(C_{2k}) \ge \left(\frac{1}{2\mu}\right)^{2k}$ proved in \eqref{eq:lowerboundhomC2k}.
    \end{proof}

    It would be useful to be able to find many pairwise colour-disjoint and vertex-disjoint paths between a given pair $(x, y)$ of vertices. To that end, we use Lemma~\ref{lem:theta} which will immediately follow from Lemma~\ref{lem:theta-general} stated below, whose proof uses the arguments of Theorem 3.7 in \cite{janzer2020rainbow}. 
    For a graph $G$, and vertices $x, y \in V(G)$, and walks $P, Q \in \Hom(P_k)$, let $PQ$ denote the closed walk in $\Hom(C_{2k}^{xy})$ obtained by concatenating $P$ and the reverse of $Q$. 

    %The following lemma shows that this can be deduced from an upper bound on $\homstar_{x, y}(C_{2k})$ in terms of $\homw(C^{2k}_{xy})$. 
   
    \begin{lemma} \label{lem:theta-general}
        Let $k, s \ge 1$ be integers. Let $G$ be a graph with a given weighting $\omega : E(G) \to \R^{> 0}$ and let $x, y \in V(G)$. Let $\mathcal B$ be a subfamily of $\Hom(C_{2k}^{xy})$ satisfying $\sum_{F\in \mathcal{B}}{\omega(F)} < \frac{1}{s^2}\homw(C_{2k}^{xy})$. Then there exist walks $P_1,\dots, P_s \in \Hom(P_k^{xy})$ such that $P_iP_j\notin \mathcal B$ for every distinct $i, j \in [s]$.
    \end{lemma}
    
    \begin{proof} Randomly and independently choose
		$s$ members $P_1,\dots, P_s$ of $\Hom(P_k^{xy})$, where for each $i \in [s]$ and $P\in \Hom(P_k^{xy})$, we have
		$\Pr(P_i=P)=\frac{\omega(P)}{\hom_w(P_k^{xy})}$. Note that this is indeed a probability distribution, and $\Hom(C_{2k}^{xy})=\{PQ: P,Q\in \Hom(P_k^{xy})\}$. 
        For any distinct $i, j \in [s]$ and any $P,Q\in \Hom(P_k^{xy})$,
		\begin{equation*}
			\Pr(P_i=P, \, P_j=Q)
			= \frac{\omega(P)}{\homw(P_k^{xy})} \cdot \frac{\omega(Q)}{\homw(P_k^{xy})} 
			=\frac{\omega(PQ)}{\hom_w(C_{2k}^{xy})}.
		\end{equation*}
        Thus, the probability that $P_iP_j\in \mathcal B$ is 
        $\frac{\sum_{F\in \mathcal{B}}{\omega(F)}}{\homw(C_{2k}^{xy})} < \frac{1}{s^2}$. Hence, by the union bound, with positive probability, $P_i P_j \notin \mathcal{B}$ for every distinct $i, j \in [s]$.
    \end{proof}
    
    \begin{lemma} \label{lem:theta}
        Let $k, s \ge 1$ be integers. Let $G$ be a properly edge-coloured graph with a given weighting $\omega : E(G) \to \R^{> 0}$ and let $x, y \in V(G)$. Suppose that $\homwstar(C_{2k}^{xy}) < \frac{1}{s^2}\homw(C_{2k}^{xy})$. Then there are $s$ pairwise colour-disjoint  and internally vertex-disjoint rainbow paths of length $k$ from $x$ to $y$.   
    \end{lemma}
    
    \begin{proof}
		By Lemma~\ref{lem:theta-general} applied to $\mathcal{B}=\Homstar(C_{2k}^{xy})$ there exist walks $P_1, \dots, P_s \in \Hom(P_{k}^{xy})$ satisfying $P_i P_j \notin \Homstar(C_{k}^{xy})$ for every distinct $i, j \in [s]$.
        In other words, $P_i P_j$ is a rainbow copy of $C_{2k}$ for every distinct $i, j \in [s]$. This means that $P_1, \dots, P_s$ are pairwise colour-disjoint and internally vertex-disjoint paths of length $k$ from $x$ to $y$, as desired. 
    \end{proof}

\section{Counting walks in expanders} \label{sec:random-walks}

    In this section we exploit the connection between the mixing time of a random walk on a graph $G$ and expansion properties of $G$. A lot of the notation and results that we use can be found in~\cite{lovasz1993random}.
    
    Suppose $G=(V,E)$ is a connected graph where $V = [n]$. 
    Consider a random walk on $V(G)$, where we start at some vertex $v_0$ and at the $i$-th step we move from $v_i$ to one of its neighbours, denoted by $v_{i+1}$, where
    each neighbour of $v_i$ is chosen as $v_{i+1}$ with probability $\frac{1}{d(v_i)}$. The sequence of vertices $(v_i)_{i \ge 0}$ defines a Markov chain. 
    Let $M$ be the $n \times n$ matrix of transition probabilities of the Markov chain, namely $M_{v, u}$ is the probability of stepping from $v$ to $u$; so $M_{v, u}= \frac{1}{d(v)}$ if $vu \in E(G)$, and $M_{v, u} = 0$ otherwise. 
    %Denote by $D$ the $n \times n$ diagonal matrix with $D_{v, v} = \frac{1}{d(v)}$ for $v \in [n]$, and let $A$ be the adjacency matrix of $G$. Then $M=D A$. 
    So the probability that a random walk starting at vertex $v$ reaches $u$ in $t$ steps is $(M^t)_{v, u}$.

    \begin{definition}
    \label{def19}
        Let $G$ be a graph on the vertex set $[n]$. 
        Let $D = D(G)$ denote the diagonal $n \times n$ matrix where  $D_{v,v}=\frac{1}{d(v)}$ for each $v\in [n]$. Let $A = A(G)$ be the adjacency matrix of $G$, let $M(G) = D A$ and let $N(G) =D^{1/2}AD^{1/2}$. Note that the matrix $N(G)$ is symmetric, so it has $n$ real eigenvalues. Let $\lam_1(N)\geq \lam_2(N)\geq \dots \geq \lam_n(N)$ denote the eigenvalues of $N:=N(G)$. 
    \end{definition}

    \begin{lemma}\label{lem:markovchain}
        Let $G$ be a bipartite graph, with a bipartition $\{X, Y\}$, on the vertex set $[n]$ with $m$ edges and no isolated vertices. 
        Let $M=D(G)A(G)$ and $N = N(G)$.
        Then for every $v,u \in V(G)$ and integer $k \ge 1$, we have
        \begin{equation*}
            \left|(M^k)_{v, u} - \frac{d(u)}{2m}\left(1 + (-1)^{k + \one(v \in X) + \one(u \in X)}\right) \right|
            \leq \sqrt{\frac{d(u)}{d(v)}} \cdot \big(\lam_2(N)\big)^k.
        \end{equation*}
    \end{lemma}
    
    Note that Lemma~\ref{lem:markovchain} says that when $k$ is even and both $v,u$ are in the same part or when $k$ is odd and $v,u$ are in different parts then   $ \left|(M^k)_{v, u} - \frac{d(u)}{m} \right|
            \leq \sqrt{\frac{d(u)}{d(v)}} \cdot \big(\lam_2(N)\big)^k.$ Note that when $k$ is even and $v$ and $u$ are in different parts or when $k$ is odd and $v$ and $u$ are in the same part then $(M^k)_{v, u}=0$. 
    \begin{proof}
        For any vector $w = (w_1, \ldots, w_n)^T$, let $\overline{w}$ be the vector $(w_1', \ldots, w_n')^T$ where $w_i' = w_i$ when $i \in X$ and $w_i' = -w_i$ when $i \in Y$. It is easy to check that $N \overline{w} = -\overline{Nw}$. Hence, if $w$ is an eigenvector of $N$ with eigenvalue $\lam$, then $\overline{w}$ is an eigenvector of $N$ with eigenvalue $-\lam$. It follows that $\lam_i = -\lam_{n+1-i}$ for $i \in [n]$. In particular, $|\lam_i| \le \lam_2$ for every $i \in \{2, \ldots, n-1\}$.
        
        One can check that $w_1$, defined as follows, is a unit eigenvector of $N$ with eigenvalue $1$. 
        \begin{equation*}
            w_1 = \frac{1}{\sqrt{2m}}\left(\sqrt{d(1)}, \sqrt{d(2)}, \dots, \sqrt{d(n)}\right)^T.
        \end{equation*}
        By the Frobenius--Perron theorem, since the entries of $N$ are non-negative and the entries of $w_1$ are positive, we have $\lam_1 = 1$. As explained above, it follows that $\overline{w_1}$ is a unit eigenvector of $N$ with eigenvalue $\lam_n = -1$. Write $w_n = \overline{w_1}$, and let $w_i$ be a unit eigenvector of $N$ with eigenvalue $\lam_i$ for $i \in \{2, \ldots, n-1\}$, such that $w_1, \ldots, w_n$ are orthogonal to each other. Note that $w_i$ is an eigenvector of $N^k$ with eigenvalue $(\lambda_i)^k$, for each $i \in [n]$.
        Since $N^k$ is a symmetric matrix and $\{w_1, \ldots, w_n\}$ is an orthonormal eigenbasis for $N^k$, we may write $N^k$ in the spectral form as follows (using $\lam_1 = 1$ and $\lam_n = -1$).
        \begin{equation*}
            N^k
            = \sum_{i=1}^n{(\lam_i)^k w_i (w_i)^T} 
            = w_1(w_1)^T + (-1)^k \,\overline{w_1}(\overline{w_1})^T + \sum_{i=2}^{n-1}{(\lam_i)^k w_i (w_i)^T}.
        \end{equation*}
 We also have $D^{1/2}ND^{-1/2} = DA = M$.
        Therefore,
        \begin{align*}
            M^k 
            &= D^{1/2}N^k D^{-1/2} \\
            &= D^{1/2} w_1 (w_1)^T  D^{-1/2} + (-1)^kD^{1/2} \,\overline{w_1} (\overline{w_1})^T  D^{-1/2} + \sum_{i=2}^{n-1}{(\lam_i)^k D^{1/2} w_i (w_i)^T D^{-1/2} }
        \end{align*}
        Let $Q=  D^{1/2} w_1 (w_1)^T  D^{-1/2} + (-1)^kD^{1/2} \,\overline{w_1} (\overline{w_1})^T  D^{-1/2}$. Then
              \begin{align*}
            M^k 
           = Q + \sum_{i=2}^{n-1}{(\lam_i)^k D^{1/2} w_i (w_i)^T D^{-1/2} }.
        \end{align*}  
        Hence
        \begin{equation}
        \label{eq:P't}
        (M^k)_{v, u} 
            = Q_{v, u} + \sum_{i=2}^{n-1}{(\lam_i)^k  w_{i, v} w_{i, u} \sqrt{\frac{d(u)}{d(v)}}}.
        \end{equation}
        
        Let $W$ be the matrix whose rows are $w_1, \ldots, w_n$. Then $W W^T = I$, implying $W^TW = I$. For each $v \in [n]$, since $(W^TW)_{v,v}=1$, we have $\sum_{i = 1}^n |w_{i, v}|^2 = 1$, so $\sum_{i = 2}^{n-1} |w_{i, v}|^2 \le 1$.
        By the Cauchy-Schwarz inequality,  $\sum_{i = 2}^{n-1} | w_{i, v}w_{i, u}| \le 
        \sqrt{\sum_{i=2}^{n-1} |w_{i,v}|^2} \sqrt{\sum_{i=2}^{n-1}  |w_{i,u}|^2}\leq 1$. 
        Since $|\lam_i| \le \lam_2$ for every $i \in \{2, \ldots, n-1\}$, the inequality in \eqref{eq:P't} implies
        \begin{equation*}
            \left| (M^k)_{v, u} - Q_{v, u} \right| 
            \le \sum_{i = 2}^{n-1} |\lam_i|^k |w_{i, u} w_{i, v}| \cdot \sqrt{\frac{d(u)}{d(v)}}
            \le (\lam_2)^k \sqrt{\frac{d(u)}{d(v)}}.
           % \qedhere
        \end{equation*}
        Finally, a straightforward calculation shows that  
        $Q_{v, u} = \frac{d(u)}{2m}\left(1 + (-1)^{k + \one(v \in X) + \one(u \in X)}\right)$ for all $v,u \in [n]$, as desired.
    \end{proof}
  
    \begin{definition} \label{def:conductance}
        For a graph $G$ with $m$ edges, let $\pi(v)= \frac{d(v)}{2m}$, and for any $S \subseteq V(G)$,  let $\pi(S) \coloneqq \sum_{s\in S}{\pi(s)}$; observe that $\pi(S) \le 1$ for every $S \subseteq V(G)$. Define the \emph{conductance} of a set $S$, denoted by $\Phi(S)$, as
        \begin{equation*}
            \Phi(S) \coloneqq \frac{e(S,S^c)}{2m \cdot \pi(S)\pi(S^c)}, 
        \end{equation*}
        and let the \emph{conductance} of a graph $G$, denoted by $\Phi_G$, be defined as
        \begin{equation*}
            \qquad \Phi_G \coloneqq \min_{S \subseteq V(G)} \Phi(S).
        \end{equation*}
    \end{definition}
   
    \begin{theorem}[Theorem 5.3 in \cite{lovasz1993random}] \label{thm:upperboundoneigenvalue} 
        Let $G$ be a graph and let $\lam_2 = \lam_2(N(G))$. Then $\lam_2 \le 1 - \frac{\Phi_G^2}{8}$.
    \end{theorem}
    
    In light of \Cref{thm:upperboundoneigenvalue} and \Cref{lem:mixing}, it will be useful to have a lower bound on $\Phi_G$ for a $(d, \eta, \eps)$-expander $G$. This is easy to achieve, as can be seen in the following lemma.
    
    \begin{lemma} \label{lem:large-phi} 
        Let $d\geq 1$, $\eta\in (0,1)$, $\varepsilon\in (0,\frac{1}{2}]$.
        Let $G$ be a $(d, \eta, \eps)$-expander on $n$ vertices. Then $\Phi_G \geq \frac{\eta}{3}$.
    \end{lemma}
    
    \begin{proof}
        Let $S\subseteq V(G)$. Since $\Phi(S)=\Phi(S^c)$, we may assume that
        $|S|\leq \frac{n}{2}\leq (1-\varepsilon) n$. Observation~\ref{obs:edge-expansion-new} thus implies  
        $e(S,S^c)\geq \frac{1}{2}\eta d|S|$.
        Let $\gamma$ be such that $e(S,S^c)=\gamma d|S|$, so $\gamma\geq \frac{\eta}{2}$. By definition of a $(d,\eta,\varepsilon)$-expander, $G$ is also $d$-minimal.
        Hence $e(G[S])\leq \frac{1}{2}d|S|$.
        Hence, $\sum_{v\in S} d(v)=2e(G[S])+e(S,S^c)\leq d|S|+\gamma d|S|$. 
        Also, observe that $\pi(S^c)\leq 1$. Hence
   
        \begin{align*}
            \Phi(S)& =\frac{e(S,S^c)}{2e(G)\pi(S)\pi(S^c)} 
            \geq \frac{e(S,S^c)}{\sum_{v\in S} d(v)}
            \geq \frac{\gamma d|S| }{d |S|+\gamma d|S|}\geq \frac{\gamma}{1+\gamma} \geq \frac{\eta}{\eta+2}\geq \frac{\eta}{3}.
        \end{align*}
        The above inequality thus implies $\Phi_G \ge \frac{\eta}{3}$.
    \end{proof}
    
	Recall that given a graph $G$ and two vertices $x,y$, the quantity $\homw(P_k^{xy})$ denotes the sum of weights of walks of length $k$ from $x$ to $y$. The following lemma and its immediate corollary will allow us to compare the values of $\homw(P_k^{xy})$, where $w(xy) = \frac{1}{\sqrt{d(x)d(y)}}$ for every edge $xy$, for different pairs of vertices $(x, y)$ in $G$. 
    
    \begin{lemma} \label{lem:mixing} 
        Let $d\geq 1$, $\eta\in (0,1)$, $\varepsilon\in (0,\frac{1}{2}]$.
        Let $G$ be a bipartite $(d, \eta, \eps)$-expander on $n$ vertices with bipartition $\{X, Y\}$. 
        Define a weighting $\omega : E(G) \to \R^{\ge 0}$,  such that for each edge $xy$ in $G$, $\omega(xy)= \frac{1}{\sqrt{d(x)d(y)}}$.        
        The following holds for any two vertices $x,y\in V(G)$
        and every integer $k\geq 1$.
        \begin{equation*}
            \left| \frac{\homw(P_k^{xy})\sqrt{d(y)}}{\sum_{z\in V(G)}{\homw(P_k^{xz})\sqrt{d(z)}}} - \frac{d(y)}{2e(G)}\left(1 + (-1)^{k + \one(x \in X) + \one(y \in X)}\right)\right| 
            \leq \sqrt{\frac{d(y)}{d(x)}}\left(1-\frac{\eta^2}{72}\right)^k.
        \end{equation*}
    \end{lemma}
    
    \begin{proof}
        Let $M=M(G)$, $N = N(G)$ and let $\lam_2$ be the second largest eigenvalue of $N$. Let $x$ be any vertex in $G$.
        Let $\mathbb{W}_{k}^x$ be a random walk of length $k$ starting at $x$.  For any walk $P=x_0\ldots x_k$ in $G$, where $x_0=x$,
		\begin{equation*}
			\Pr[\mathbb{W}_{k}^x=P]=\frac{1}{d(x_0)}\cdot \ldots \cdot \frac{1}{d(x_{k-1})}.
		\end{equation*}
		Note that
		\begin{align*}
			\omega(P)
			& = \frac{1}{\sqrt{d(x_0)d(x_1)}} \cdot \ldots \cdot \frac{1}{\sqrt{d(x_{k-1})d(x_k)}} \\
			& = \frac{\sqrt{d(x_0)}}{\sqrt{d(x_k)}}\cdot \frac{1}{d(x_0) \cdot \ldots \cdot d(x_{k-1})}
			=\frac{\sqrt{d(x_0)}}{\sqrt{d(x_k)}}\cdot \Pr[\mathbb{W}_{k}^x=P]. 
		\end{align*}
         
		For any vertex $y$ in $G$, by definition,
		\def \bW {\mathbb{W}}
		\begin{align*}
			\Pr[\mathbb{W}_k^x \text{ ends at $y$}]
			= \sum_{P:=xx_1\ldots x_{k-1}y}\Pr(\bW_k^x = P)
			=\frac{\sqrt{d(y)}}{\sqrt{d(x)}}\sum_{P:P=xx_1\ldots x_{k-1}y}\omega(P) 
			=\homw(P_k^{xy})\cdot\sqrt{\frac{d(y)}{d(x)}}. 
		\end{align*}
        On the other hand, notice that
		\begin{equation*}
			1=\sum_{z\in V(G)}{ \Pr[\mathbb{W}_k^x \textit{ ends at } z]}=\sum_{z\in V(G)}{\homw(P_k^{xz})\cdot\sqrt{\frac{d(z)}{d(x)}}}.
		\end{equation*}
        
        Recall from our discussion before Definition~\ref{def19}, that the probability that a random walk of length $k$ starting at $x$ ends at $y$ is exactly $(M^k)_{x,y}$, thus
		\begin{equation*}
			(M^k)_{x,y} 
			=\Pr[\mathbb{W}_k^x \textit{ ends at } y]
			=\frac{\homw(P^{xy}_k)\sqrt{\frac{d(y)}{d(x)}}}{\sum_{z\in V(G)}{\homw(P_k^{xz})\sqrt{\frac{d(z)}{d(x)}}}}
			=\frac{\homw(P^{xy}_k)\sqrt{d(y)}}{\sum_{z\in V(G)}{\homw(P^{xz}_k)\sqrt{d(z)}}}.
		\end{equation*}
        By Lemma~\ref{lem:markovchain}, 
        \begin{equation}
        \label{eqlem18}
             \left| \frac{\sqrt{d(y)}\homw(P^{xy}_k)}{\sum_{z\in V(G)}{\sqrt{d(z)}\homw(P^{xz}_k)}}-\frac{d(y)}{2e(G)}\left(1 + (-1)^{k + \one(x \in X) + \one(y \in X)}\right)\right| 
            \leq \sqrt{\frac{d(y)}{d(x)}} \cdot (\lam_2)^k.
        \end{equation}
        \Cref{thm:upperboundoneigenvalue} gives that $\lam_2 \le 1 - \frac{\Phi_G^2}{8}$, and by \Cref{lem:large-phi} we have $\Phi_G \ge \frac{\eta}{3}$. It follows that $\lam_2 \leq 1-\frac{\eta^2}{72}$. 
        Combining this inequality with \eqref{eqlem18}, the lemma follows.
    \end{proof}
    
	\begin{corollary} \label{cor:mixing} 
		Let $d\geq 1$, $\eta\in (0,1)$, $\varepsilon\in (0,\frac{1}{2}]$. Let $G$ be a bipartite $(d, \eta, \eps)$-expander on $n$ vertices with a bipartition $\{X, Y\}$ and with a given weighting $\omega : E(G) \to \R^{> 0}$ such that for each edge $xy$ in $G$, $\omega(xy)= \frac{1}{\sqrt{d(x)d(y)}}$. The following holds for any two vertices $x,y\in X$
		and even integer $k\geq 2$.
		\begin{equation*}
			\left|\frac{\homw(P^{xy}_k)\sqrt{d(y)}}{\sum_{z\in X}{\homw(P^{xz}_k)\sqrt{d(z)}}} - \frac{d(y)}{e(G)}\right| 
			\leq \sqrt{n}\left(1-\frac{\eta^2}{72}\right)^k.
		\end{equation*}
	\end{corollary}
    \begin{proof} 
		Let $x,y\in X$ be given. 
		Since $k$ is even, for each $z\in Y$, $\Hom(P_k^{xz})=\emptyset$. The claim follows by applying Lemma~\ref{lem:mixing} and 
		the fact that $\sqrt{\frac{d(y)}{d(x)}}\leq \sqrt{n}$.
    \end{proof}
    
	The next lemma contains the main takeaway from our discussion about random walks in almost regular bipartite expanders. It tells us that for relatively large $k$, and for the weighting defined by $w(xy) = \frac{1}{\sqrt{d(x)d(y)}}$, the values of $\homw(C_{2k}^{xy})$ do not differ by much over the range of pairs $(x, y)$ where $x$ and $y$ are in the largest part of the bipartition.

    \begin{lemma} \label{lem:rhomin-rhomax}
        Let $n, d, \mu, k \ge 1$ and $\eta \in (0, 1)$, $\varepsilon\in (0,\frac{1}{2}]$.
        Suppose that $k$ is an even integer satisfying $k \ge \frac{2^9 \log n}{\eta^2}$.
        Let $G$ be a bipartite $(d, \eta, \eps)$-expander with maximum degree at most $\mu d$ and with a given weighting $\omega : E(G) \to \R^{> 0}$ such that for each edge $xy$ in $G$, $\omega(xy)= \frac{1}{\sqrt{d(x)d(y)}}$. Let $\{X, Y\}$ be the bipartition of $G$ and suppose that $|X| \ge \frac{n}{2}$. 
        Let 
        \begin{align*}
            \rhomin 
            & = \min \{ \homw(C_{2k}^{xy}) \,:\, x, y \in X \} \\
            \rhomax
            & = \max \{ \homw(C_{2k}^{xy}) \,:\, x, y \in X \}.
        \end{align*}
        Then $\rhomax \le 2^{10} \mu^2 \cdot \rhomin$.
    \end{lemma}
    
    \begin{proof}
        By \Cref{cor:mixing}, for every $x, y \in X$ we have 
        \begin{align*}
            \left| \frac{\homw(P_k^{xy}) \sqrt{d(y)}}{\sum_{z \in X}\homw(P^{xz}_k)\sqrt{d(z)}} - \frac{d(y)}{e(G)} \right| 
            & \le \sqrt{n}\left(1 - \frac{\eta^2}{72}\right)^k \\
            & \le \sqrt{n} \cdot \exp\left(-\frac{k\eta^2}{72}\right)
            \le \sqrt{n} \cdot \exp(-4 \log n) 
            \le \frac{1}{n^3}
            \le \frac{d(y)}{2e(G)}.
        \end{align*}
        It follows that
        \begin{equation*}
            \frac{d(y)}{2e(G)}
            \le \frac{\homw(P_k^{xy})\sqrt{d(y)}}{\sum_{z \in X}\homw(P^{xz}_k)\sqrt{d(z)}}
            \le \frac{2d(y)}{e(G)}.
        \end{equation*}

        Writing the same for $x, w\in X$ we obtain 
        \begin{equation*}
            \frac{d(w)}{2e(G)}
            \le \frac{\homw(P_k^{xw})\sqrt{d(w)}}{\sum_{z \in X}\homw(P_k^{xz})\sqrt{d(z)}}
            \le \frac{2d(w)}{e(G)}.
        \end{equation*}
        
        Hence, every $x, y, w \in X$ satisfy 
		\begin{equation*}
			\frac{\homw(P^{xy}_k)}{\homw(P^{xw}_k)} 
			= \frac{\homw(P^{xy}_k)\sqrt{d(y)}}{\homw(P^{xw}_k)\sqrt{d(w)}} \cdot \sqrt{\frac{d(w)}{d(y)}}
			\le \frac{4d(y)}{d(w)} \cdot \sqrt{\frac{d(w)}{d(y)}}
			= 4\sqrt{\frac{d(y)}{d(w)}}
			\le 4\sqrt{2\mu},
		\end{equation*}
        where the second inequality follows from the fact that the minimum degree of $G$ is at least $\frac{d}{2}$ (by Observation~\ref{obs:d-expand}) and the maximum degree is at most $\mu d$. 
        
        Observe that $\homw(P_k^{xy}) = \homw(P_k^{yx})$ for $x, y \in X$. It follows that, for every $x, y, z, w \in X$,
        \begin{equation*}
            \frac{\homw(P_k^{xy})}{\homw(P_k^{zw})}
            = \frac{\homw(P_k^{xy})}{\homw(P_k^{xz})} \cdot \frac{\homw(P_k^{xz})}{\homw(P_k^{zw})} 
            \le 32 \mu.
        \end{equation*}
        Let $x, y, z, w \in X$ satisfy $\rho_{\max} = \homw(C_{2k}^{xy})$ and $\rho_{\min} = \homw(C_{2k}^{zw})$.
        Then 
        \begin{equation*}
            \frac{\rhomax}{\rhomin}
            = \frac{\homw(C_{2k}^{xy})}{\homw(C_{2k}^{zw})}
            = \left(\frac{\homw(P_k^{xy})}{\homw(P_k^{zw})}\right)^2
            \le 2^{10} \mu^2,
        \end{equation*}
        proving the lemma.
    \end{proof}
    
    Recall that, by \Cref{lem:theta}, pairs of vertices $(x, y)$ for which $\homwstar(C_{2k}^{xy})$ is considerably smaller than $\homw(C_{2k}^{xy})$ can be used to build many colour-disjoint  and internally vertex-disjoint rainbow paths. 
    The following lemma shows that for large enough $k$ and $d$, almost all pairs of vertices in one of the parts of an almost regular bipartite expander satisfy this property.

    \begin{lemma} \label{lem:few-bad-pairs}
        Let $n, d, \mu, k, s, p\ge 1$ and $\eta \in (0, 1)$, $\varepsilon\in (0,\frac{1}{2}]$.
        Suppose that $k$ is even and satisfies $k \ge \frac{2^9  \log n}{\eta^2}$ and that $d \ge 2^{38} k^3 \mu^7 s^4 p^2 n^{1/k}$.
        Let $G$ be a bipartite $(d, \eta, \varepsilon)$-expander on $n$ vertices with maximum degree at most $\mu d$ and with a given weighting $\omega : E(G) \to \R^{> 0}$ such that for each edge $xy$ in $G$, $\omega(xy)= \frac{1}{\sqrt{d(x)d(y)}}$. Let $\{X, Y\}$ be the bipartition of $G$ and suppose that $|X| \ge \frac{n}{2}$.
        Then for all but at most $\frac{n^2}{p}$ pairs $(x, y)$ with $x,y \in X$ the following holds.   
        \begin{equation*}
            \homwstar(C_{2k}^{xy}) \le \frac{1}{s^2} \homw(C_{2k}^{xy}).
        \end{equation*}
    \end{lemma}
    
    \begin{proof}
        Let $S = 2^{11} \mu^2 s^2 p$, and let $\rhomin$ and $\rhomax$ be defined as in the statement of \Cref{lem:rhomin-rhomax}. Then, by the same lemma we have $\rhomax \le 2^{10} \mu^2 \rhomin$.
        Let $A$ be the collection of (ordered) pairs $(x, y)$ with $x, y \in X$ that satisfy  $\homwstar(C_{2k}^{xy}) \ge \frac{1}{s^2}\homw(C_{2k}^{xy})$.
        Then
        \begin{align*}
            \homwstar(C_{2k}) 
            \ge \sum_{(x, y) \in A} \homwstar(C_{2k}^{xy}) 
            \ge \frac{1}{s^2} \sum_{(x, y) \in A} \homw(C_{2k}^{xy})
            \ge \frac{|A| \cdot \rhomin}{s^2}.
        \end{align*}
      
        Note that $\sum_{x, y \in X} \homw(C_{2k}^{xy}) = \sum_{x,y \in Y} \homw(C_{2k}^{xy})$, and so $\homw(C_{2k}) = 2\sum_{x, y \in X} \hom(C_{2k}^{xy})$.
        Hence, by \Cref{lem:non-rainbow-hom} (which is applicable since $d \ge 2^{16} \mu^3 k^3 S^2 n^{1/k}$), 
        \begin{equation*}
            \homwstar(C_{2k}) 
            \le \frac{1}{S} \cdot \homw(C_{2k})
            = \frac{2}{S} \sum_{x, y \in X} \homw(C_{2k}^{xy}) 
            \le \frac{2n^2 \cdot \rhomax}{S}.
        \end{equation*}
        Combining the lower and upper bounds on $\homwstar(C_{2k})$, we obtain the required inequality, as follows.
        \begin{equation*}
            |A|
            \le \frac{2n^2 \cdot s^2 \cdot \rhomax}{S \cdot \rhomin}
            \le \frac{n^2 \cdot 2^{11} \mu^2 s^2}{S} 
            = \frac{n^2}{p},
            \qedhere
        \end{equation*}
        as desired.
    \end{proof}

\section{Rainbow paths and subdivisions in expanders} \label{sec:main-proof}    

    We now prove our first main result about rainbow subdivisions.
    
    \begin{proof}[Proof of \Cref{thm:main}.]  
     Let $G'$ be a bipartite subgraph of $G$ with at least $\frac{e(G)}{2}$ edges. Then $d(G') \ge (\log n)^{53}$.  We apply \Cref{lem:bounded-max-deg-expander} to obtain a subgraph $H$ of $G'$ with the following properties.
        \begin{enumerate}
            \item 
                $H$ is a $(d, \eta, \varepsilon)$-expander on $n'$ vertices, where $d \ge \frac{d(G')}{2500(\log n)^2} \ge (\log n)^{50}$, $\eps=\frac{1}{2}$ and $\eta = \frac{\eps}{100(\log n')^2} = \frac{1}{200(\log n')^2}$,
            \item
                $H$ has maximum degree at most $\mu d$, where $\mu = 2500(\log n')^2$.
        \end{enumerate}
        
        Denote the bipartition of $H$ by $\{X, Y\}$, and suppose that $|X| \ge \frac{n'}{2}$. 
        Next, we aim to apply Lemma~\ref{lem:few-bad-pairs} to $H$. With that in mind, we
        let $k$ be the smallest even integer which is at least $\frac{2^9 \log n'}{\eta^2}$, so $k = \Theta\!\left((\log n')^{5}\right)$. 
        Let $s=2\binom{m}{2} k$ and $p = 8m$. 
        Let $A$ be the set of pairs $(x, y)$ with $x, y \in X$ that satisfy $\homwstar(C_{2k}^{xy}) > \frac{1}{s^2}\homw(C_{2k}^{xy})$.
        Note that $2^{38} k^3 \mu^7 s^4 p^2 (n')^{1/k} = \Theta\!\left((\log n')^{49}\right) < d$ and so, by \Cref{lem:few-bad-pairs}, we have $|A| \le \frac{(n')^2}{p} =\frac{(n')^2}{8m}< \frac{1}{m-1}\binom{|X|}{2}$,
        noting that $n'$ is sufficiently large.
            
        Hence by Tur\'an's theorem, there is a subset $Z \subseteq X$ of size $m$, such that $(x, y) \notin A$ for every distinct $x, y \in Z$. 
        By \Cref{lem:theta}, for any two vertices $x$ and $y$ in $Z$, there exist $s$ many pairwise colour-disjoint and internally vertex-disjoint rainbow paths of length $k$ from $x$ to $y$. By the choice of $s$, one can find greedily $\binom{m}{2}$ many paths of length $k$, which are pairwise colour-disjoint and internally  vertex-disjoint, and are internally vertex-disjoint from $Z$, each of which connecting a different pair of vertices in $Z$. This gives us the desired rainbow subdivision of $K_m$.
    \end{proof}

    Given a set of $m$ vertices in a graph $G$, a \emph{$K_m$-subdivision rooted at $Z$} is a subgraph consisting of $\binom{m}{2}$ paths, each joining a different pair of distinct vertices in $Z$, whose interiors are pairwise vertex-disjoint and disjoint from $Z$. By slightly adapting the proof of
    \Cref{thm:main}, one can show that any bipartite $(d, \eta, \eps)$-expander $G$ (with suitable parameters) contains a rainbow $K_m$-subdivision rooted at $Z$ for almost all the $m$-sets $Z$ in $V(G)$ (not just in $X$). 
     By using some additional tools, we next show that in a bipartite $(d,\eta,\eps)$-expander with suitable parameters, in fact, one can find  a rainbow $K_m$-subdivision rooted at $Z$ for every $m$-set $Z$ in $V(G)$.

    \begin{theorem} \label{thm:main-rooted} 
        Let $n, L, d, \mu, m \geq 2$, $\eta\in (0,1)$ and $\eps\in (0,\frac{1}{8}]$, and suppose that $L = \frac{2^{10} \log n}{\eta^2}$ and $d \ge \frac{2^{122} m^8 \mu^7 (\log n)^7}{\eta^{14}}$. Let $G$ be a bipartite $(d, \eta, \eps)$-expander with maximum degree at most $\mu d$, and let $Z$ be a set of $m$ vertices in $G$. Then there is  a rainbow $K_m$-subdivision, rooted at $Z$, where every edge is subdivided at most $L$ times.
    \end{theorem}
    
    \begin{proof}
        Let $\ell = \frac{4\log n}{\eta}$ and let $k$ be the smallest even integer satisfying $k \ge \frac{2^9 \log n}{\eta^2}$. One can check that $L \ge 2(\ell + 1) + k$.
        
        \begin{claim} \label{claim:rainbow-connect}
        Let $M$ be any set of colours and vertices such that $|M|\leq 2\binom{m}{2}(L+1) $. Let $x,y$ be any two vertices in $G$. There exists a rainbow $x,y$-path of length at most $L$ in $G$ that avoids $M$.    
        \end{claim}
        \begin{proof}[Proof of Claim~\ref{claim:rainbow-connect}]

        Let $q = 256\ell$. 
        By \Cref{lem:reachable-robust-new} (using $d \ge \frac{20q\ell + 8[2\binom{m}{2}(L+1)]}{\eta}$), there exists a subset $U_{x} \subseteq V(G)$ of size at least $(1-\eps)n$ and a collection of paths $\mathcal P = \{P(u) : u \in U_x \}$, where for each $u \in U_x$ the path $P(u)$ is a rainbow path from $x$ to $u$ of length at most $\ell+1$ that avoids $M$, and no colour appears in more than $\frac{n}{q}$ of the paths in $\mathcal P$. Similarly, there exists a subset $U_{y} \subseteq V(G)$ of size at least $(1-\eps)n$ and a collection of paths $\mathcal Q = \{Q(u) : u \in U_y\}$ where for each $u \in U_y$ the path $Q(u)$ is a rainbow path from $y$ to $u$ of length at most $\ell+1$ that avoids $M$, and no colour appears in  more than $\frac{n}{q}$ of the paths in $\mathcal Q$. Write $U = U_x \cap U_y$; then $|U| \ge (1 - 2\eps)n \ge \frac{3n}{4}$.
        
        We call an ordered pair $(u,v)$ with $u, v\in U$ \emph{colour-bad} if there is a colour that appears on both paths $P(u)$ and $Q(v)$. We next show that the number of colour-bad pairs in $U$ is small compared to the number of all pairs. Indeed, let $H$ be the auxiliary graph on the vertex set $U$ where $uv$ is an edge whenever at least one of $(u, v)$ and $(v, u)$ is colour-bad.
        Note that $d_H(u)\leq \frac{2(\ell+1)n}{q}$, for every $u \in U$, since $P(u)$ has length at most $\ell+1$ and any colour on $P(u)$ appears at most $\frac{n}{q}$ times in the collection $\mathcal{Q}$ (and similarly with the roles of $P$ and $Q$ reversed). Thus, $e(H) \leq \frac{(\ell+1)n}{q}|U| \le \frac{2\ell n}{q}|U| \leq \frac{n^2}{128}$.
        
        Let $s = |M| + 2(\ell + 1) + 1$; then $s \le 2\binom{m}{2}(L+1) + 2\ell + 3 \le 2 m^2 L$.
        Denote the bipartition of $G$ by $\{X, Y\}$ and suppose that $|X| \ge \frac{n}{2}$.
        Call a pair $(u,v)$, with $u, v \in X$, \emph{$s$-bad} if $\homwstar(C_{2k}^{uv}) > \frac{1}{s^2} \hom(C_{2k}^{uv})$. 
        Applying \Cref{lem:few-bad-pairs} with $p = 64$ and verifying the condition on $d$ (namely, that $d \ge 2^{38} k^3 \mu^7 s^4 p^2 (n)^{1/k}$), at most $\frac{n^2}{64}$ ordered pairs $(u, v)$, with $u, v \in X$, are $s$-bad. 
        
        We claim that there is a pair $(u, v)$, with $u, v \in U \cap X$, which is neither colour-bad nor $s$-bad. Indeed, the total number of ordered pairs $(u, v)$ with $u, v \in U \cap X$ where $(u, v)$ is either colour-bad or $s$-bad is at most $\frac{n^2}{128}+\frac{n^2}{64}\leq \frac{n^2}{32}$. Since $|U| \ge \frac{3n}{4}$ and $|X| \ge \frac{n}{2}$, we have $|X \cap U| \ge \frac{n}{4}$, so the number of ordered pairs $(u, v)$ with $u, v \in U \cap X$ is certainly more than $\frac{n^2}{32}$. Hence, there is a pair $(u, v)$ which is neither colour-bad nor $s$-bad, as claimed. 
        Since $(u,v)$ is not $s$-bad, by Lemma~\ref{lem:theta} there are $s$ many pairwise colour-disjoint and internally vertex-disjoint rainbow paths of length $k$ from $u$ to $v$. By the choice of $s$, there is at least one such path $T(uv)$ which shares no colours with the paths $P(u)$ and $Q(v)$ and also avoids $M$. Since $(u,v)$ is not colour-bad, $P(u)$ and $Q(v)$ are colour-disjoint. It follows that $P(u)T(uv)Q(v)$   is a rainbow $x,y$-walk avoiding $M$ which contains a rainbow $x,y$-path, as desired.
        \end{proof}
                
        Let $(x_1, y_1), \ldots, (x_{\binom{m}{2}}, y_{\binom{m}{2}})$ be an arbitrary ordering of the unordered pairs $(x, y)$ where $x, y \in Z$ and $x \neq y$.
        We iteratively build paths $P_i$ for $i \in [\binom{m}{2}]$, as follows. Let $P_1$ be any rainbow $x_1,y_1$-path of length at most $L$, which exists
        by Claim~\ref{claim:rainbow-connect}. In general, suppose $P_1,\dots, P_i$ have been defined, where $i<\binom{m}{2}$. We let $M_i$ denote the set of vertices  and colours used in $\cup_{j=1}^i P_j$ and let $P_{i+1}$ be a rainbow $x_{i+1},y_{i+1}$-path of length at most $L$ that avoids $M_i\setminus\{x_{i+1},y_{i+1}\}$.
        Since $|M_i|\leq 2\binom{m}{2}(L+1)$, by Claim~\ref{claim:rainbow-connect} such a path $P_{i+1}$ exists. Hence, we are able to find $P_1,\dots, P_{\binom{m}{2}}$
        as described above. Now, $\bigcup_{i=1}^{\binom{m}{2}} P_i$ forms a rainbow $K_m$-subdivision rooted at $Z$ in which each edge is subdivided
        at most $L$ times.
    \end{proof}

\section{Conclusion} \label{sec:conc} 
	In this paper we showed that there is a constant $c\leq 53$ such that for fixed $m$ and sufficiently large $n$ any  $n$-vertex properly edge-coloured graph $G$ 
	with at least $n(\log n)^c$ edges contains a rainbow subdivision of $K_m$. 
	On the other hand, an immediate lower bound is given by the best known lower bound from \cite{keevash2007rainbow} on $\exstar(n,\C)$, which is $\Omega(n\log{n})$. This shows that our bound is tight up to a polylogarithmic factor.
	We pose the following question.

	\begin{question} \label{question:correct-bound}
		Fix $m\geq 2$. What is the smallest $c$ such that for % some $q$ and 
		all sufficiently large $n$ the following holds: if $G$ is a properly edge-coloured graph on $n$ vertices with at least $\Omega(n(\log{n})^c)$ edges, then it contains a rainbow subdivision of $K_m$? 
		In particular, is $c=1$?
	\end{question}

	For the clarity of presentation, we did not optimise our arguments to obtain
	the best possible value of $c$. However,
	to answer Question~\ref{question:correct-bound}, new ideas will be needed. %For instance, it is possible that using a more refined expansion idea, our bound can be improved.
	Note that even the correct order of magnitude of $\exstar(n,\C)$ is still unknown.   
	On another note, as mentioned previously, Janzer~\cite{janzer2020rainbow} proved $\exstar(n,\C)=O(n(\log n)^4)$, using the counting lemmas on closed walks.   It is worth to mention that from our methods, an alternative proof of $\ex^{*}(n,\mathcal{C})=O(n(\log{n})^5)$  can be obtained by using the basic expansion property of an expander (Lemma~\ref{lem:reachable-robust-new}) and a digraph idea used by Letzter~\cite{letzter2021tight} regarding the Tur\'an number of the family of tight cycles. 

	\subsection*{Acknowledgements}
		We would like to thank the anonymous referee for useful comments, in particular for pointing out an error in the previous version of our proof which we have now rectified.
		Since the submission of the first version of our paper to arXiv, several interesting developments regarding rainbow Tur\'an numbers of cycles and of clique subdivisions occurred.
		First, Tomon \cite{tomon2022robust} showed that $\ex^*(n, \C) \le n (\log n)^{2 + o(1)}$, and that every $n$-vertex graph with at least $n (\log n)^{6 + o(1)}$ edges contains a rainbow subdivision of any constant sized clique, using a novel argument about sampling in expanders.
		The latter bound was subsequently improved to $n (\log n)^{2 + o(1)}$ by Wang \cite{wang}.
		The bound on the rainbow Tur\'an number of cycles was later improved to $\ex^*(n, \C) \le O(n (\log n)^2)$ by Janzer and Sudakov \cite{janzer2022tur} and independently by Kim, Lee, Liu, and Tran \cite{joonkyung}.
	%In forthcoming work by Alon, Buci\'c, Sauermann, Zakharov, and Zamir, the last bound is further improved to the almost tight bound of $O(n \log n \log \log n)$.

\bibliography{ref.bib}
\bibliographystyle{amsplain}

\end{document}